\documentclass[12pt,a4paper]{article}
\usepackage{latexsym,amssymb,amsfonts,amsmath,amsthm,nccmath,float,enumitem,setspace,bm,authblk}
\usepackage[usenames,dvipsnames]{xcolor}
\usepackage[hidelinks]{hyperref}
\usepackage[margin=2cm]{geometry}

\hypersetup{
    colorlinks,
    linkcolor={red!80!black},
    citecolor={blue!80!black},
    urlcolor={blue!80!black}
}

\newtheorem{Theorem}{Theorem}[section]

\newtheorem{Lemma}{Lemma}[section]
\newtheorem{Example}{Example}[section]

\newtheorem{Construction}{Construction}[section]

\newcommand{\Z}{\mathbb{Z}}

\def \leq {\leqslant}
\def \geq {\geqslant}

\let\oldproofname=\proofname
\renewcommand{\proofname}{\rm\bf{\oldproofname}}

\begin{document}

\title{Existence of magic rectangle sets over finite abelian groups}

\author[a]{Shikang Yu}
\author[a]{Tao Feng \thanks{Supported by NSFC under Grant 12271023}}
\author[a]{Hengrui Liu}
\affil[a]{School of Mathematics and Statistics, Beijing Jiaotong University, Beijing, 100044, P.R. China}
\renewcommand*{\Affilfont}{\small\it}
\renewcommand\Authands{ and }

\affil[ ]{healthyu@bjtu.edu.cn, tfeng@bjtu.edu.cn, henryleo@bjtu.edu.cn }
\date{}

\maketitle
\begin{abstract}
Let $a$, $b$ and $c$ be positive integers. Let $(G,+)$ be a finite abelian group of order $abc$. A $G$-magic rectangle set MRS$_G(a,b;c)$ is a collection of $c$ arrays of size $a\times b$ whose entries are elements of a group $G$, each appearing exactly once, such that the sum of each row in every array equals a constant $\gamma\in G$ and the sum of each column in every array equals a constant $\delta\in G$. This paper establishes the necessary and sufficient conditions for the existence of an MRS$_G(a,b;c)$ for any finite abelian group $G$, thereby confirming a conjecture presented by Cichacz and Hinc.
\end{abstract}



\section{Introduction}\label{sec:intro}

Let $m$ be a positive integer. A {\em magic square} of order $m$ is an $m\times m$ array with entries from the set $\{1,2,\ldots,m^2\}$, each appearing exactly once, such that the sum of the $m$ numbers in each row, each column, and each of the two main diagonals is the same. A magic square of order $m$ exists for any positive integer $m\neq 2$. We refer the reader to \cite{km} for a brief history and  constructions for magic squares.

Magic rectangles are a natural generalization of magic squares. Let $a$ and $b$ be positive integers. A {\em magic rectangle} MR$(a,b)$ is an $a\times b$ array with entries from the set $\{1,2,\ldots,ab\}$, each appearing exactly once,  such that the sum of the entries in each row is a constant $\alpha$, and the sum of the entries in each column is another constant $\beta$. The total sum of all entries in an MR$(a,b)$ is $a\alpha=b\beta=1+2+\cdots+ab=ab(ab+1)/2$. Thus $\alpha=b(ab+1)/2$ and $\beta=a(ab+1)/2$. Since $\alpha$ and $\beta$ are both integers, $a$ and $b$ must have the same parity. Harmuth \cite{Har1,Har2} settled the existence problem of magic rectangles as early as 1881.

\begin{Theorem}\label{Thm:MR} {\rm \cite{Har1,Har2}}
For $a,b>1$, there exists an MR$(a,b)$ if and only if $a\equiv b\pmod{2}$ and $(a,b)\neq (2,2)$.
\end{Theorem}

Following the work by Bier and Rogers \cite{br}, Bier and Kleinschmidt \cite{bk} provided a simple proof for Theorem \ref{Thm:MR} by utilizing centrally symmetric rectangles. Subsequently, Hagedorn \cite{Hage-1} further simplified Bier and Kleinschmidt's method, resulting in an entirely theoretical proof that does not depend on any concrete example of magic rectangles given in \cite[Section 2.7]{bk}.

Motivated by the need to schedule a special kind of incomplete round robin tournaments, called {\em handicap tournaments}, Froncek \cite{Fro1} proposed a generalization of magic rectangles, which he termed magic rectangle sets. Let $a$, $b$ and $c$ be positive integers. A {\em magic rectangle set} MRS$(a,b;c)$ is a collection of $c$ arrays of size $a\times b$, whose entries are elements of $\{1,2,\ldots,abc\}$, each appearing exactly once, such that the sum of each row in every $a\times b$ array equals a constant $b(abc+1)/2$ and the sum of each column in every $a\times b$ array equals another constant $a(abc+1)/2$. Froncek \cite{Fro2} settled the existence problem for magic rectangle sets.

\begin{Theorem}\emph{\cite[Theorem 3.2]{Fro2}}
For $a,b>1$, there exists an MRS$(a,b;c)$ if and only if either $abc\equiv 1\pmod{2}$ or $a\equiv b\equiv 0\pmod{2}$, and $(a,b)\neq (2,2)$.
\end{Theorem}

Let $a$, $b$ and $c$ be positive integers. Let $(G,+)$ be a finite abelian group of order $abc$. A {\em $G$-magic rectangle set} MRS$_G(a,b;c)$ is a collection of $c$ arrays of size $a\times b$ with entries from the group $G$, each appearing exactly once, such that all row sums in every $a\times b$ array equal a constant $\gamma\in G$ and all column sums in every $a\times b$ array equal another constant $\delta\in G$. When $c=1$, an MRS$_G(a,b;1)$ is called a {\em $G$-magic rectangle}, and is simply written as an MR$_G(a,b)$.

Denote by $\Z_v$ the additive group of integers modulo $v$. Evans \cite[Theorem 1]{Evans} investigated the existence of an MR$_{\Z_v}(a,b)$, which he termed a {\em modular magic rectangle}. Cichacz and Hinc \cite{CH21-1} generalized Evans's result to any finite abelian group.

\begin{Theorem}\label{Thm:MR-G} {\rm \cite[Theorem 3.11]{CH21-1}}
For any abelian group $G$ of order $ab$ with $a,b>1$, an MR$_G(a,b)$ exists if and only if either $a\equiv b\pmod{2}$, or the Sylow $2$-subgroup of $G$ is noncyclic.
\end{Theorem}

Furthermore, Cichacz and Hinc \cite{CH21-1} examined the existence of a $G$-magic rectangle set for any finite abelian group $G$. They established the following necessary condition for the existence of an MRS$_G(a,b;c)$ (see \cite[Observations 1.6, 1.7 and 3.8]{CH21-1}). We outline the proof for completeness.

\begin{Lemma} \label{lem:necessary} \emph{\cite{CH21-1}}
Let $G$ be an abelian group of order $abc$ with $a,b>1$. If there exists an MRS$_G(a,b;c)$, then either $a\equiv b\equiv 0 \pmod{2}$, or the Sylow $2$-subgroup of $G$ is trivial or noncyclic, and when $2\in\{a,b\}$, $ab\equiv 0\pmod{4}$.
\end{Lemma}

\begin{proof}
If the Sylow $2$-subgroup of $G$ is cyclic and nontrivial, then $G$ has the unique element of order $2$, denoted by $\theta$. Let $M$ be an MRS$_G(a,b;c)$. Let $\gamma$ and $\delta$ be the row sum and the column sum of an $a\times b$ array in $M$, respectively. Then
$ca\gamma=cb\delta=\sum_{g\in G}g=\theta$. Thus $(ca-1)b\cdot ca\gamma=(ca-1)b\cdot \theta$ and $(cb-1)a\cdot cb\delta=(cb-1)a\cdot \theta$. Since $|G|=abc$ and $a,b>1$, we have $b\cdot \theta=0$ and $a\cdot \theta=0$, which implies that $a\equiv b\equiv 0 \pmod{2}$. Therefore, if there is an MRS$_G(a,b;c)$, then either the Sylow $2$-subgroup of $G$ is trivial or noncyclic, or $a\equiv b\equiv 0 \pmod{2}$.

Furthermore, consider the case of $a=2$ and $b=2l+1$ for some positive integer $l$. Note that $G\cong S_2\oplus H$, where $S_2$ is the Sylow $2$-subgroup of $G$ and $H$ is the largest subgroup of odd order of $G$. If there exists an MRS$_{S_2\oplus H}(2,2l+1;c)$ with  $\gamma=(\gamma_1,\gamma_2)$ and $\delta=(\delta_1,\delta_2)$ as the row sum and the column sum, respectively, where $\gamma_1,\delta_1\in S_2$ and $\gamma_2,\delta_2\in H$, then $(2l+1)\delta_1=2\gamma_1$. Since $\gcd(2l+1,2)=1$, there exists $x\in S_2$ such that $\delta_1=2x$. Since $|H|$ is odd, there exists $y\in H$ such that $\delta_2=2y$. It follows that $\delta=(\delta_1,\delta_2)=2(x,y)$. Without loss of generality, we can assume that the element in the first row and the first column in the first $2\times (2l+1)$ array is $(x,y)$, but then the element in the second row and the first column in this array has to be $(x,y)$ as well, a contradiction. Therefore, if there exists an MRS$_G(a,b;c)$, and one of $a$ and $b$ is $2$, then the other cannot be odd.
\end{proof}

Cichacz and Hinc \cite[Conjecture 5.1]{CH21-1} conjectured that the necessary conditions for the existence of an MRS$_G(a,b;c)$ given in Lemma \ref{lem:necessary} are also sufficient. They confirmed this conjecture when there do not exist positive integers $l$ and $n$ such that either $a=2l+1$ and $b=2^n$, or $a=2^n$ and $b=2l+1$.

\begin{Lemma}\label{thm:CH1} \emph{\cite[Theorem 3.9]{CH21-1}}
For any abelian group $G$ of order $abc$, when $a,b>1$ and $\{a,b\}\neq\{2l+1,2^n\}$ for some positive integers $l$ and $n$, an MRS$_G(a,b;c)$ exists if and only if either $a\equiv b\equiv 0 \pmod{2}$ or the Sylow $2$-subgroup of $G$ is trivial or noncyclic.
\end{Lemma}

This paper is devoted to examining all of the remaining cases not covered by Lemma \ref{thm:CH1}. We prove the following result.

\begin{Theorem}\label{thm:main}
Let $G$ be an abelian group of order $abc$ with $a,b>1$. An MRS$_G(a,b;c)$ exists if and only if either $a\equiv b\equiv 0 \pmod{2}$, or the Sylow $2$-subgroup of $G$ is trivial or noncyclic, and when $2\in\{a,b\}$, $ab\equiv 0\pmod{4}$.
\end{Theorem}

In Section \ref{sec:Problem Reduction}, we reduce the existence problem of an MRS$_G(a,b;c)$ over a finite abelian group $G$ to that of an MRS$_{\Z_p\oplus S_2}(p,4,|S_2|/4)$, where $p$ is an odd prime and $S_2$ is an abelian and noncyclic 2-group. The {\em exponent} of a finite group $G$, denoted by $\exp(G)$, is the least common multiple of the orders of all elements of $G$. In Section \ref{sec:exp4}, we establish the existence of an MRS$_{\Z_p\oplus S_2}(p,4,|S_2|/4)$ for the case when the exponent of $S_2$ is no more than 4. In Section \ref{sec:exp8}, we establish the existence of an MRS$_{\Z_p\oplus S_2}(p,4,|S_2|/4)$ for the case when the exponent of $S_2$ is larger than 4. Section \ref{sec:con} concludes this paper.

%
%

\section{Problem reduction}\label{sec:Problem Reduction}

By Lemmas \ref{lem:necessary} and \ref{thm:CH1}, in order to prove Theorem \ref{thm:main}, it remains to examine the existence of an MRS$_G(2l+1,2^n;c)$, where $l\geq 1$ and $n\geq 2$. In this section, we reduce the existence problem of an MRS$_G(2l+1,2^n;c)$ to that of an MRS$_{\Z_p\oplus S_2}(p,4,|S_2|/4)$, where $p$ is a prime factor of $2l+1$ and $S_2$ is the Sylow $2$-subgroup of $G$.

The following two constructions play an important role in our reduction of the problem.

\begin{Construction}\label{con:directsum}{\rm \cite[Lemma 3.3]{CH21-1}}
Let $G_1$ be an abelian group of order $abc_1$. Let $G_2$ be an abelian group of order $c_2$, whose Sylow $2$-subgroup is either trivial or noncyclic. If there exists an MRS$_{G_1}(a,b;c_1)$, then there exists an MRS$_{G_1\oplus G_2}(a,b;c_1c_2)$.
\end{Construction}

\begin{Construction}\label{con:odd}{\rm \cite[Lemma 3.4]{CH21-1}}
Let $H$ be a finite abelian group. Let $k$ and $t$ be positive integers such that $t$ is a divisor of $2k+1$. Let $\langle t\rangle$ be the subgroup of order $(2k+1)/t$ in $\Z_{2k+1}$. If there exists an MRS$_{\langle t\rangle\oplus H}(a,b;c)$, then there exists an MRS$_{\Z_{2k+1}\oplus H}(a,b;ct)$.
\end{Construction}

\begin{Lemma}\label{lem:odd}
Let $a\geq 3$ be an odd integer and $n\geq 2$ be an integer. Let $G$ be an abelian group of order $2^nac$, whose Sylow $2$-subgroup $S_2$ is noncyclic. Let $p$ be a prime factor of $a$. If there exists an MRS$_{\Z_{p}\oplus S_2}(p,4;|S_2|/4)$, then there exists an MRS$_G(a,2^n;c)$.
\end{Lemma}

\begin{proof}
For $n\geq 2$, if there exists an MRS$_G(a,4;2^{n-2}c)$, which consists of $2^{n-2}c$ arrays, each of size $a\times 4$, then we can select any $2^{n-2}$ of these arrays to form an array of size $a\times 2^n$. Consequently, by considering $c$ distinct ways to choose such a subset of arrays, we obtain an MRS$_G(a,2^n;c)$. Therefore, in order to establish the existence of an MRS$_G(a,2^n;c)$, it suffices to examine the existence of an MRS$_G(a,4;2^{n-2}c)$.

Furthermore, since $p$ is a prime factor of $a$, if there exists an MRS$_{G}(p,4;2^{n-2}ac/p)$, which consists of $2^{n-2}ac/p$ arrays, each of size $p\times 4$, then we can select any $a/p$ of these arrays to form an array of size $a\times 4$. Consequently, by considering $2^{n-2} c$ distinct ways to choose such a subset of arrays, we obtain an MRS$_G(a,4;2^{n-2} c)$. Therefore, establishing the existence of an MRS$_{G}(p,4;2^{n-2}ac/p)$ is sufficient to derive the existence of an MRS$_G(a,4;2^{n-2}c)$.

Note that $G\cong \Z_{p_1^{\alpha_1}}\oplus \cdots\oplus \Z_{p_r^{\alpha_r}}\oplus S_2$, where $r\geq 1$, $p_1,\ldots,p_r$ are odd primes and $\alpha_1,\ldots,\alpha_r$ are positive integers. Without loss of generality, assume that $p_1=p$ and $\alpha_1=\alpha$. Applying Construction \ref{con:odd} with an MRS$_{\Z_{p}\oplus S_2}(p,4;|S_2|/4)$, which exists by assumption, we obtain an MRS$_{\Z_{p^{\alpha}}\oplus S_2}(p,4;p^{\alpha-1}|S_2|/4)$. Then apply Construction \ref{con:directsum} to obtain an MRS$_{G}(p,4;|G|/4p)$, i.e., an MRS$_{G}(p,4;2^{n-2}ac/p)$, as desired.
\end{proof}

By Lemmas \ref{lem:necessary}, \ref{thm:CH1} and \ref{lem:odd}, to prove Theorem \ref{thm:main}, it suffices to establish the existence of an MRS$_{\Z_{p}\oplus S_2}(p,4;|S_2|/4)$ for any odd prime $p$ and any abelian and noncyclic $2$-group $S_2$.

In what follows we always assume that $S_2$ is an abelian and noncyclic $2$-group.


\section{The case of $\exp(S_2)\leq 4$}\label{sec:exp4}

Recall that the {\em exponent} of a finite group $G$, denoted by $\exp(G)$, is the least common multiple of the orders of all elements of $G$. For a nonnegative integer $\alpha$ and a finite group $G$, denote by $G^{\alpha}$ the direct product of $\alpha$ copies of $G$, where $G^0$ is understood to be the trivial group. 

Since $S_2$ is an abelian and noncyclic $2$-group, if $\exp(S_2)\leq 4$, then $S_2\cong (\Z_4)^\alpha\oplus (\Z_2)^\beta$ for some integers $\alpha,\beta\geq 0$ and $\alpha+\beta\geq 2$. In this section, we establish the existence of an MRS$_{\Z_{p}\oplus (\Z_4)^\alpha\oplus (\Z_2)^\beta}(p,4;2^{2\alpha+\beta-2})$, where $p$ is an odd prime.

We use the notation MRS$_G^*(a,b;c)$ to represent an MRS$_G(a,b;c)$ where the sum of each row and the sum of each column in every $a\times b$ array are both zero.

\begin{Lemma}\label{lem:p,2,2}
Let $p$ be an odd prime and $G\cong \Z_p\oplus \Z_2\oplus \Z_2$. Then there exists an MRS$_G^*(p,4;1)$.
\end{Lemma}

\begin{proof}
It is readily checked that
$$
{\left(
    \begin{array}{cccc}
(1,0,0)  & (0,1,0)    & (1,0,1)  & (1,1,1) \\
(0,1,1)  & (2,0,0)    & (2,0,1)  & (2,1,0) \\
(2,1,1)  & (1,1,0)    & (0,0,0)  & (0,0,1) \\
    \end{array}
    \right)}
$$
is an MRS$_{\Z_3\oplus \Z_2\oplus \Z_2}^*(3,4;1)$. For $p\geq 5$, it follows from the proof of \cite[Lemma 3.7]{CH21-1} that there exists an MRS$_{\Z_p\oplus \Z_2\oplus \Z_2}^*(p,4;1)$.
\end{proof}

\begin{Lemma}\label{lem:p,2,2,2}\emph{\cite[Lemma 2.1]{CH21-2}}
Let $p$ be an odd prime and $G\cong \Z_p\oplus (\Z_2)^3$. Then there exists an MRS$_G^*(p,4;2)$.
\end{Lemma}

\begin{Lemma}\label{lem:p,2,4}
Let $p$ be an odd prime and $G\cong \Z_p\oplus \Z_2\oplus \Z_4$. Then there exists an MRS$_G^*(p,4;2)$.
\end{Lemma}

\begin{proof}
By \cite[Lemma 2.2]{CH21-2}, there exists an MRS$_{\Z_3\oplus \Z_2\oplus \Z_4}^*(3,4;2)$. Note that $\Z_p\oplus \Z_2\oplus \Z_4\cong \Z_{2p}\oplus \Z_4$. For $p\geq5$, an MRS$_{\Z_{2p}\oplus \Z_4}(p,4;2)$  with row sum $(0,0)$ and column sum $(0,(p-5)/2)$ was constructed in \cite[Lemma 2.2]{CH21-2}. We can make a minor modification to this construction such that all row sums and all column sums are $(0,0)$.

The first 5 rows in the two $p\times 4$ arrays of an MRS$_{\Z_{2p}\oplus \Z_4}(p,4;2)$ are listed below:
$$
A_1={\left(
    \begin{array}{cccc}
(1,0)  & (0,3)    & (p-1,0)  & (p,1) \\
(p-1,1)  & (p,3)    & (1,3)  & (0,1) \\
(p,2)  & (p+2,3)    & (0,2)  & (2p-2,1) \\
(2,3)  & (2p-2,3)    & (2p-1,1)  & (1,1) \\
(2p-2,2)  & (0,0)    & (p+1,2)  & (p+1,0) \\
    \end{array}
    \right)}
$$
and
$$
A_2={\left(
    \begin{array}{cccc}
(2,0)  & (p-2,0)    & (2p-1,3)  & (p+1,1) \\
(p-1,2)  & (p+2,0)    & (p-2,3)  & (p+1,3) \\
(1,2)  & (2p-1,2)    & (p+2,2)  & (p-2,2) \\
(p,0)  & (2p-1,0)    & (p-1,3)  & (2,1) \\
(2p-2,0)  & (2,2)    & (p+2,1)  & (p-2,1) \\
    \end{array}
    \right).}
$$
Note that $A_1$ and $A_2$ cover each element in the set $\{(i,j)\mid i\in I, j\in\Z_4\}$ exactly once, where $I=\{0,1,2,p-2,p-1,p,p+1,p+2,2p-2,2p-1\}$, and all row sums and all column sums in $A_1$ and $A_2$ are $(0,0)$. Since all the elements in $\Z_{2p}\setminus I$ can be partitioned into $p-5$ pairs, each of the form $\{x,-x\}$, we can create two rows in one $p\times 4$ array of an MRS$_{\Z_{2p}\oplus \Z_4}^*(p,4;2)$ as follows:
$$
{\left(
    \begin{array}{cccc}
(x,0)  & (-x,2)    & (x,1)  & (-x,1) \\
(-x,0)  & (x,2)    & (-x,3)  & (x,3) \\
    \end{array}
    \right).}
$$
This procedure provides us an MRS$_{\Z_{2p}\oplus \Z_4}^*(p,4;2)$.
\end{proof}

\begin{Construction}\label{con:2,2}
Let $G_1\cong \Z_p\oplus H\oplus \Z_{2^{\alpha-1}}\oplus \Z_{2^{\beta-1}}$ and $G_2\cong \Z_p\oplus H\oplus \Z_{2^{\alpha}}\oplus \Z_{2^{\beta}}$, where $\alpha,\beta\in\{1,2\}$, $p$ is an odd prime, and $H$ is a finite abelian group. If there exists an MRS$_{G_1}^*(p,4;|G_1|/(4p))$, then there exists an MRS$_{G_2}^*(p,4;|G_2|/(4p))$.
\end{Construction}

\begin{proof}
Let $G_1=\{(w,u,v) \mid w\in\Z_p\oplus H, u\in\Z_{2^{\alpha-1}}, v\in\Z_{2^{\beta-1}}\}$ and $G_2=\{(w,x,y) \mid w\in\Z_p\oplus H, x\in\Z_{2^{\alpha}}, y\in\Z_{2^{\beta}}\}$.
Let $M_1,M_2,\cdots,M_{|G_1|/(4p)}$ be an MRS$_{G_1}^*(p,4;|G_1|/(4p))$, where $M_s=(m^s_{ij})$, $1\leq s\leq|G_1|/(4p)$, $1\leq i\leq p$ and $1\leq j\leq 4$. Write $$m^s_{ij}=(w^s_{ij},u^s_{ij},v^s_{ij}),$$
where $w^s_{i,j}\in \Z_p\oplus H$, $u^s_{i,j}\in \Z_{2^{\alpha-1}}$ and $v^s_{i,j}\in \Z_{2^{\beta-1}}$.

For every $1\leq s\leq |G_1|/(4p)$, we construct four arrays of size $p\times 4$ defined on $G_2$, say $M_{s1}=(m^{s1}_{ij})$, $M_{s2}=(m^{s2}_{ij})$, $M_{s3}=(m^{s3}_{ij})$ and $M_{s4}=(m^{s4}_{ij})$, $1\leq i\leq p$ and $1\leq j\leq 4$, as follows: when $p\equiv 3 \pmod 4$, we take
$$m^{s1}_{ij}=
(w^s_{ij},2u^s_{ij},2v^s_{ij});$$
$$m^{s2}_{ij}=\left\{
\begin{array}{lll}
m^{s1}_{ij}+(0,0,1), &  \text{if } i=1; \\
m^{s1}_{ij}+(0,1,0), &  \text{if } i=2; \\
m^{s1}_{ij}+(0,-1,-1), &  \text{if } i=3; \\
m^{s1}_{ij}+(0,0,1), &  \text{otherwise};
\end{array}
\right.\ \ \ \
m^{s3}_{ij}=\left\{
\begin{array}{lll}
m^{s1}_{ij}+(0,1,0), &  \text{if } i=1; \\
m^{s1}_{ij}+(0,-1,-1), &  \text{if } i=2; \\
m^{s1}_{ij}+(0,0,1), &  \text{if } i=3; \\
m^{s1}_{ij}+(0,1,0), &  \text{otherwise};
\end{array}
\right.
$$
and
$$m^{s4}_{ij}=\left\{
\begin{array}{lll}
m^{s1}_{ij}+(0,-1,-1), &  \text{if } i=1; \\
m^{s1}_{ij}+(0,0,1), &  \text{if } i=2; \\
m^{s1}_{ij}+(0,1,0), &  \text{if } i=3; \\
m^{s1}_{ij}+(0,-1,-1), &  \text{otherwise};
\end{array}
\right.
$$
when $p\equiv 1 \pmod 4$, we take
$$m^{s1}_{ij}=\left\{
\begin{array}{lll}
(w^s_{ij},2u^s_{ij},2v^s_{ij})+(0,1,0), & \text{if } i=1; \\
(w^s_{ij},2u^s_{ij},2v^s_{ij})+(0,0,1), & \text{if } i=2; \\
(w^s_{ij},2u^s_{ij},2v^s_{ij})+(0,-1,-1), & \text{if } i=3; \\
(w^s_{ij},2u^s_{ij},2v^s_{ij}),  &   \text{otherwise};
\end{array}
\right.
$$
$$m^{s2}_{ij}=\left\{
\begin{array}{lll}
(w^s_{ij},2u^s_{ij},2v^s_{ij}), & \text{if } i=1; \\
(w^s_{ij},2u^s_{ij},2v^s_{ij})+(0,1,0), & \text{otherwise};
\end{array}
\right.
$$
$$m^{s3}_{ij}=\left\{
\begin{array}{lll}
(w^s_{ij},2u^s_{ij},2v^s_{ij}), & \text{if }  i=2; \\
(w^s_{ij},2u^s_{ij},2v^s_{ij})+(0,0,1), &   \text{otherwise};
\end{array}
\right.
$$
and
$$m^{s4}_{ij}=\left\{
\begin{array}{lll}
(w^s_{ij},2u^s_{ij},2v^s_{ij}), & \text{if } i=3; \\
(w^s_{ij},2u^s_{ij},2v^s_{ij})+(0,-1,-1), &   \text{otherwise}.
\end{array}
\right.
$$
Then the set ${\cal M}=\{M_{s1}, M_{s2}, M_{s3}, M_{s4}\mid 1\leq s\leq |G_1|/(4p)\}$ is an MRS$_{G_2}^*(p,4;|G_2|/(4p))$.

Note that every element $(w,x,y)$ of $G_2$ has a unique representation as $(w,x,y)=(w,2u,2v)+z$, where $(w,u,v)\in G_1$ and $z\in\{(0,0,0),(0,0,1),(0,1,0),(0,-1,-1)\}$.
Since every element of $G_1$ appears exactly once in $M_1,M_2,\cdots, M_{|G_1|/(4p)}$, it is readily checked that every element of $G_2$ appears exactly once in the arrays from $\cal M$. Since $|\Z_{2^{\alpha}}|$ and $|\Z_{2^{\beta}}|$ are both no more than $4$, it is readily checked that all the row sums and all the column sums in every array in $\cal M$ are zero.
\end{proof}

\begin{Lemma}\label{lem:exp4}
Let $G\cong \Z_p\oplus S_2$, where $p$ is an odd prime and $S_2\cong (\Z_4)^\alpha\oplus (\Z_2)^\beta$ for some integers $\alpha,\beta\geq 0$ and $\alpha+\beta\geq 2$. Then there exists an MRS$_G^*(p,4;|G|/(4p))$.
\end{Lemma}

\begin{proof}
We proceed by induction on $|S_2|$. If $|S_2|\in\{4,8\}$, then $S_2$ is isomorphic to one of $(\Z_2)^2$, $(\Z_2)^3$ and $\Z_4\oplus \Z_2$. By Lemmas \ref{lem:p,2,2}, \ref{lem:p,2,2,2} and \ref{lem:p,2,4}, there exists an MRS$_G^*(p,4,|G|/(4p))$.

Assume that $|S_2|\geq 16$. Note that $S_2\cong H\oplus \Z_{2^i}\oplus \Z_{2^j}$ for some $i,j\in\{1,2\}$ and some abelian group $H$. Let $G_1\cong \Z_p\oplus H\oplus \Z_{2^{i-1}}\oplus \Z_{2^{j-1}}$. By induction hypothesis there exists an MRS$^*_{G_1}(p,4,|G_1|/(4p))$. Apply Construction \ref{con:2,2} to obtain an MRS$^*_{G_2}(p,4,|G_2|/(4p))$, where $G_2\cong \Z_p\oplus H\oplus \Z_{2^{i}}\oplus \Z_{2^{j}}\cong G$.
\end{proof}

\section{The case of $\exp(S_2)\geq 8$}\label{sec:exp8}

In this section, we establish the existence of an MRS$_{\Z_{p}\oplus S_2}(p,4;|S_2|/4)$ for any odd prime $p$ and any abelian and noncyclic $2$-group $S_2$ with $\exp(S_2)\geq 8$. We introduce the notion of incomplete magic rectangle sets to construct magic rectangle sets.

Let $a$, $b$, $c$ and $c'$ be positive integers. Let $G$ be an abelian group of order $abc$. Let $\mathcal{H}$ be a set of some subgroups of $G$ such that $\bigcup_{H\in {\cal H}}H$ consists of $abc'$ distinct elements of $G$. A {\em $(G,\mathcal{H})$-incomplete magic rectangle set}, written as IMRS$_{G\setminus (\bigcup_{H\in {\cal H}}H)}(a,b;c-c')$, is a collection of $c-c'$ arrays of size $a\times b$ with entries from $G\setminus (\bigcup_{H\in {\cal H}}H)$, each appearing exactly once, such that all row sums in every array are equal to a constant $\gamma\in G$ and all column sums in every array are equal to another constant $\delta\in G$. We use the notation IMRS$^*_{G\setminus (\bigcup_{H\in {\cal H}}H)}(a,b;c-c')$ to represent an IMRS$_{G\setminus (\bigcup_{H\in {\cal H}}H)}(a,b;c-c')$ where the sum of each row and the sum of each column in every $a\times b$ array are both zero.


Let $H$ be a subgroup of a group $G$, and let $p$ be an odd prime. In what follows, we always simply write $\Z_p\oplus(G\setminus H)$ instead of $(\Z_p\oplus G)\setminus (\Z_p\oplus H)$.

\begin{Example}\label{eg:IMRS}
There exists an IMRS$^*_{\Z_6\oplus (\Z_8\setminus \{0,4\})}(3,4;3)$ as follows:
$$
{\left(
    \begin{array}{cccc}
(0,1)  & (0,3)   & (1,1)  & (5,3) \\
(0,2)  & (0,6)   & (3,1)  & (3,7) \\
(0,5)  & (0,7)   & (2,6)  & (4,6)
    \end{array}
    \right)},\ \ \
{\left(
    \begin{array}{cccc}
(1,2)  & (3,2)   & (5,1)  & (3,3) \\
(1,7)  & (2,1)   & (5,5)  & (4,3) \\
(4,7)  & (1,5)   & (2,2)  & (5,2)
    \end{array}
    \right)},
$$
and
$$
{\left(
    \begin{array}{cccc}
(1,3)  & (4,1)   & (5,7)  & (2,5) \\
(2,7)  & (4,5)   & (5,6)  & (1,6) \\
(3,6)  & (4,2)   & (2,3)  & (3,5)
    \end{array}
    \right)}.
$$
\end{Example}

We shall present three constructions for $(G,\mathcal{H})$-incomplete magic rectangle sets to establish existences for $G$-magic rectangle sets.


\subsection{Construction I}

The following construction is simple but very useful.

\begin{Construction}\label{con:Hole}
Let $G$ be a finite abelian group with a subgroup $H$. Suppose that there exists an IMRS$^*_{G\setminus H}(a,b;|G\setminus H|/(ab))$. If there exists an MRS$^*_H(a,b;|H|/(ab))$, then there exists an MRS$^*_G(a,b;|G|/(ab))$.
\end{Construction}

\begin{Lemma}\label{lem:p,2,8}
For any odd prime $p$, there exist an IMRS$^*_{\Z_{2p}\oplus (\Z_8\setminus \{0,4\})}(p,4;3)$ and an MRS$^*_{\Z_{2p}\oplus \Z_8}(p,$ $4;4)$.
\end{Lemma}

\begin{proof}
When $p=3$, by Example \ref{eg:IMRS}, there exists an IMRS$^*_{\Z_6\oplus (\Z_8\setminus \{0,4\})}(3,4;3)$. When $p\geq5$, the first 5 rows in the three $p\times 4$ arrays of an IMRS$^*_{\Z_{2p}\oplus (\Z_8\setminus \{0,4\})}(p,4;3)$ are listed below:
$$
A_1={\left(
    \begin{array}{cccc}
(p,2)  & (0,2)    & (0,7)  & (p,5) \\
(p+2,3)  & (p-2,2)    & (p+1,1)  & (p-1,2) \\
(1,1)  & (1,3)    & (p-2,5)  & (p,7) \\
(2p-2,3)  & (p+1,6)    & (2,2)  & (p-1,5) \\
(2p-1,7)  & (0,3)    & (2p-1,1)  & (2,5) \\
    \end{array}
    \right)},
$$

$$
A_2={\left(
    \begin{array}{cccc}
(0,5)  & (p,3)    & (2p-2,1)  & (p+2,7) \\
(p-1,6)  & (0,6)    & (1,6)  & (p,6) \\
(2,7)  & (p,1)    & (0,1)  & (p-2,7) \\
(p+1,7)  & (p-1,1)    & (2,6)  & (2p-2,2) \\
(2p-2,7)  & (p+1,5)    & (2p-1,2)  & (p+2,2) \\
    \end{array}
    \right)},
$$
and
$$
A_3={\left(
    \begin{array}{cccc}
(p+1,3)  & (2p-1,3)    & (2,1)  & (p-2,1) \\
(2,3)  & (p-2,3)    & (2p-2,5)  & (p+2,5) \\
(p-1,7)  & (p+2,1)    & (1,2)  & (2p-2,6) \\
(2p-1,5)  & (p+2,6)    & (p-2,6)  & (1,7) \\
(2p-1,6)  & (p-1,3)    & (p+1,2)  & (1,5) \\
    \end{array}
    \right).}
$$
Note that $A_1$, $A_2$ and $A_3$ cover each element in the set $\{(i,j)\mid i\in I, j\in\Z_8\setminus\{0,4\}\}$ exactly once, where $I=\{0,1,2,p-2,p-1,p,p+1,p+2,2p-2,2p-1\}$, and all row sums and all column sums in $A_1$, $A_2$ and $A_3$ are $(0,0)$. Since all the elements in $\Z_{2p}\setminus I$ can be partitioned into $(p-5)/2$ quadruples, each of the form $\{x,y,-x,-y\}$, we can create two rows for each of the three $p\times 4$ arrays of an IMRS$^*_{\Z_{2p}\oplus (\Z_8\setminus \{0,4\})}(p,4;3)$ as follows:
$$
B_1=\left(
\begin{array}{cccc}
(x,1)  & (-x,1)    & (y,7)  & (-y,7) \\
(-x,7)  & (x,7)    & (-y,1)  & (y,1) \\
\end{array}
\right),
$$
$$
B_2=\left(
\begin{array}{cccc}
(x,2)  & (-x,2)    & (y,6)  & (-y,6) \\
(-x,6)  & (x,6)    & (-y,2)  & (y,2) \\
\end{array}
\right),
$$
and
$$
B_3=\left(
\begin{array}{cccc}
(x,3)  & (-x,3)    & (y,5)  & (-y,5) \\
(-x,5)  & (x,5)    & (-y,3)  & (y,3) \\
\end{array}
\right).
$$
This procedure provides us an IMRS$^*_{\Z_{2p}\oplus (\Z_8\setminus \{0,4\})}(p,4;3)$.

Furthermore, apply Construction \ref{con:Hole} with an MRS$^*_{\Z_{2p}\oplus \Z_2}(p,4,1)$, which exists by Lemma \ref{lem:p,2,2}, to obtain an MRS$^*_{\Z_{2p}\oplus \Z_8}(p,4,4)$.
\end{proof}

\begin{Lemma}\label{lem:p,4,8}
For any odd prime $p$, there exist an IMRS$^*_{\Z_{4p}\oplus (\Z_8\setminus\{0,4\})}(p,4;6)$ and an MRS$^*_{\Z_{4p}\oplus \Z_8}(p,$ $4;8)$.
\end{Lemma}

\begin{proof}
When $p=3$, we construct an IMRS$^*_{\Z_{12}\oplus (\Z_8\setminus \{0,4\})}(3,4;6)$ as follows:
$$
\left(
\begin{array}{cccc}
(2,2)  & (9,6)   & (7,2)  & (6,6) \\
(3,5)  & (11,1)   & (1,3)  & (9,7) \\
(7,1)  & (4,1)   & (4,3)  & (9,3)
\end{array}
\right), \ \ \
\left(
\begin{array}{cccc}
(0,5)  & (1,5)   & (7,7)  & (4,7) \\
(1,6)  & (2,1)   & (6,3)  & (3,6) \\
(11,5)  & (9,2)   & (11,6)  & (5,3)
\end{array}
\right),
$$
$$
\left(
\begin{array}{cccc}
(9,1)  & (7,3)   & (10,1)  & (10,3) \\
(8,1)  & (5,7)   & (9,5)  & (2,3) \\
(7,6)  & (0,6)   & (5,2)  & (0,2)
\end{array}
\right),\ \ \
\left(
\begin{array}{cccc}
(5,5)  & (11,3)   & (2,7)  & (6,1) \\
(5,6)  & (10,2)   & (10,6)  & (11,2) \\
(2,5)  & (3,3)   & (0,3)  & (7,5)
\end{array}
\right),
$$
$$
\left(
\begin{array}{cccc}
(1,1)  & (2,6)   & (3,2)  & (6,7) \\
(1,2)  & (6,5)   & (11,7)  & (6,2) \\
(10,5)  & (4,5)   & (10,7)  & (0,7)
\end{array}
\right),\ \ \
\left(
\begin{array}{cccc}
(8,6)  & (4,2)   & (4,6)  & (8,2) \\
(8,3)  & (8,5)   & (5,1)  & (3,7) \\
(8,7)  & (0,1)   & (3,1)  & (1,7)
\end{array}
\right).
$$
When $p\geq5$, the first 5 rows in the six $p\times 4$ arrays of an IMRS$^*_{\Z_{4p}\oplus (\Z_8\setminus \{0,4\})}(p,4;6)$ are listed below:
$$
A_1=\left(
    \begin{array}{cccc}
(2p-1,1)  & (3p-2,3)    & (3p+1,2)  & (2,2) \\
(2p-2,5)  & (p-1,7)    & (3p+2,5)  & (2p+1,7) \\
(0,5)  & (2p+1,6)    & (2p-1,2)  & (0,3) \\
(p+1,6)  & (2,3)    & (p-2,2)  & (2p-1,5) \\
(3p+2,7)  & (2p,5)    & (3p,5)  & (4p-2,7) \\
    \end{array}
    \right),
$$
$$
A_2=\left(
    \begin{array}{cccc}
(4p-2,2)  & (p+2,2)    & (p+2,5)  & (2p-2,7) \\
(2,6)  & (3p-2,2)    & (p+1,7)  & (4p-1,1) \\
(2p,7)  & (p-2,3)    & (p,1)  & (2,5) \\
(2p,2)  & (3p+1,6)    & (3p-1,2)  & (0,6) \\
(0,7)  & (1,3)    & (2p-2,1)  & (2p+1,5) \\
    \end{array}
    \right),
$$
$$
A_3=\left(
    \begin{array}{cccc}
(4p-2,5)  & (p+2,1)    & (2p-2,3)  & (p+2,7) \\
(4p-2,6)  & (2,7)    & (p-1,2)  & (3p+1,1) \\
(1,7)  & (p-1,6)    & (2p,6)  & (p,5) \\
(p+1,5)  & (p-1,3)    & (2p+1,2)  & (4p-1,6) \\
(3p+2,1)  & (p-2,7)    & (p+2,3)  & (3p-2,5) \\
    \end{array}
    \right),
$$
$$
A_4=\left(
    \begin{array}{cccc}
(2p+2,2)  & (3p-2,6)    & (3p,6)  & (0,2) \\
(3p+2,6)  & (p-1,1)    & (3p-1,7)  & (p,2) \\
(2p,1)  & (3p-1,3)    & (1,1)  & (3p,3) \\
(3p-2,1)  & (2p+2,3)    & (3p+1,5)  & (4p-1,7) \\
(2p-2,6)  & (3p+2,3)    & (3p-1,5)  & (1,2) \\
    \end{array}
    \right),
$$
$$
A_5=\left(
    \begin{array}{cccc}
(0,1)  & (4p-2,3)    & (2p+2,1)  & (2p,3) \\
(p+2,6)  & (p-2,5)    & (3p-1,6)  & (3p+1,7) \\
(p-2,1)  & (p+1,3)    & (p+1,1)  & (p,3) \\
(1,5)  & (1,6)    & (3p-2,7)  & (p,6) \\
(2p-1,3)  & (2p+2,7)    & (3p,1)  & (p-1,5) \\
    \end{array}
    \right),
$$
and
$$
A_6=\left(
    \begin{array}{cccc}
(p-2,6)  & (2p+2,6)    & (4p-1,2)  & (p+1,2) \\
(2p+2,5)  & (2p-1,6)    & (2p-2,2)  & (2p+1,3) \\
(4p-1,5)  & (3p,2)    & (3p+2,2)  & (2p-1,7) \\
(2p+1,1)  & (4p-1,3)    & (3p-1,1)  & (3p+1,3) \\
(3p,7)  & (p,7)    & (2,1)  & (4p-2,1) \\
    \end{array}
    \right).
$$
Note that $A_1,A_2,\ldots,A_6$ cover each element in the set $\{(i,j)\mid i\in I, j\in\Z_8\setminus\{0,4\}\}$ exactly once, where $I=\{0,1,2,p-2,p-1,p,p+1,p+2,2p-2,2p-1,2p,2p+1,2p+2,3p-2,3p-1,3p,3p+1,3p+2,4p-2,4p-1\}$, and all row sums and all column sums in $A_1,A_2,\ldots,A_6$ are $(0,0)$. Since all the elements in $\Z_{4p}\setminus I$ can be partitioned into $(p-5)/2$ octuples, each of the form $\{a,b,c,d,-a,-b,-c,-d\}$, we can create two rows for each of the six $p\times 4$ arrays of an IMRS$^*_{\Z_{4p}\oplus (\Z_8\setminus \{0,4\})}(p,4;6)$ as follows:
$$
B_1=\left(
    \begin{array}{cccc}
(a,1)  & (-a,1)    & (b,7)  & (-b,7) \\
(-a,7)  & (a,7)    & (-b,1)  & (b,1) \\
    \end{array}
    \right),
$$
$$
B_2=\left(
    \begin{array}{cccc}
(a,2)  & (-a,2)    & (b,6)  & (-b,6) \\
(-a,6)  & (a,6)    & (-b,2)  & (b,2) \\
    \end{array}
    \right),
$$
$$
B_3=\left(
    \begin{array}{cccc}
(a,3)  & (-a,3)    & (b,5)  & (-b,5) \\
(-a,5)  & (a,5)    & (-b,3)  & (b,3) \\
    \end{array}
    \right),
$$
$$
B_4=\left(
    \begin{array}{cccc}
(c,1)  & (-c,1)    & (d,7)  & (-d,7) \\
(-c,7)  & (c,7)    & (-d,1)  & (d,1) \\
    \end{array}
    \right),
$$
$$
B_5=\left(
    \begin{array}{cccc}
(c,2)  & (-c,2)    & (d,6)  & (-d,6) \\
(-c,6)  & (c,6)    & (-d,2)  & (d,2) \\
    \end{array}
    \right),
$$
and
$$
B_6=\left(
    \begin{array}{cccc}
(c,3)  & (-c,3)    & (d,5)  & (-d,5) \\
(-c,5)  & (c,5)    & (-d,3)  & (d,3) \\
    \end{array}
    \right).
$$
This procedure provides us an IMRS$^*_{\Z_{4p}\oplus (\Z_8\setminus \{0,4\})}(p,4;6)$.

Furthermore, apply Construction \ref{con:Hole} with an MRS$^*_{\Z_{4p}\oplus \Z_2}(p,4,2)$, which exists by Lemma \ref{lem:p,2,4}, to obtain an MRS$^*_{\Z_{4p}\oplus \Z_8}(p,4;8)$.
\end{proof}

\subsection{Construction II}

The following construction is a variation of Construction \ref{con:Hole}.

\begin{Construction}\label{con:Hole-1}
Let $H_1$ and $H_2$ be two subgroups of a finite abelian group $G$ and $H_1\cap H_2=H$. Suppose that there exists an IMRS$^*_{G\setminus (H_1\cup H_2)}(a,b;|G\setminus (H_1\cup H_2)|/(ab))$. If there exist an IMRS$^*_{H_1\setminus H}(a,b;|H_1\setminus H|/(ab))$ and an MRS$^*_{H_2}(a,b;|H_2|/(ab))$, then there exists an MRS$^*_G(a,b;|G|/(ab))$.
\end{Construction}

\begin{Lemma}\label{lem:p,8,8}
For any odd prime $p$, there exists an MRS$^*_{\Z_p\oplus \Z_8\oplus \Z_8}(p,4;16)$.
\end{Lemma}

\begin{proof}
Let $H_1\cong \Z_p\oplus \Z_2\oplus \Z_8$ and $H_2\cong \Z_p\oplus \Z_8\oplus \Z_2$ be two distinct subgroups of $G=\Z_p\oplus \Z_8\oplus \Z_8$. Let $H=H_1\cap H_2\cong \Z_p\oplus \Z_2\oplus \Z_2$.

Firstly, we construct an IMRS$^*_{G\setminus (H_1\cup H_2)}(p,4;9)$. The first 3 rows in the nine $p\times 4$ arrays of an IMRS$^*_{G\setminus (H_1\cup H_2)}(p,4;9)$ are listed below:
$$
A_1=\left(
    \begin{array}{cccc}
(0,3,3)  & (0,3,5)  & (0,5,3) &  (0,5,5) \\
(1,6,6)  & (1,6,2)  & (p-1,2,6) &  (p-1,2,2) \\
(p-1,7,7)  & (p-1,7,1) & (1,1,7)  &  (1,1,1)
    \end{array}
    \right),
$$
$$
A_2=\left(
    \begin{array}{cccc}
(0,3,2)  & (0,3,6)  & (0,5,2) &  (0,5,6) \\
(1,6,1)  & (1,6,7)  & (p-1,2,1) &  (p-1,2,7) \\
(p-1,7,5)  & (p-1,7,3) & (1,1,5)  &  (1,1,3)
    \end{array}
    \right),
$$
$$
A_3=\left(
    \begin{array}{cccc}
(0,3,1)  & (0,3,7)  & (0,5,1) &  (0,5,7) \\
(1,6,5)  & (1,6,3)  & (p-1,2,5) &  (p-1,2,3) \\
(p-1,7,2)  & (p-1,7,6) & (1,1,2)  &  (1,1,6)
    \end{array}
    \right),
$$
$$
A_4=\left(
    \begin{array}{cccc}
(0,2,3)  & (0,2,5)  & (0,6,3) & (0,6,5)  \\
(p-1,1,6)  & (p-1,1,2)  & (1,7,6) & (1,7,2)  \\
(1,5,7)  & (1,5,1)  & (p-1,3,7) & (p-1,3,1)
    \end{array}
    \right),
$$
$$
A_5=\left(
    \begin{array}{cccc}
(0,2,2)  & (0,2,6)  & (0,6,2) & (0,6,6)  \\
(p-1,1,1)  & (p-1,1,7)  & (1,7,1) & (1,7,7)  \\
(1,5,5)  & (1,5,3)  & (p-1,3,5) & (p-1,3,3)
    \end{array}
    \right),
$$
$$
A_6=\left(
    \begin{array}{cccc}
(0,2,1)  & (0,2,7)  & (0,6,1) & (0,6,7)  \\
(p-1,1,5)  & (p-1,1,3)  & (1,7,5) & (1,7,3)  \\
(1,5,2)  & (1,5,6)  & (p-1,3,2) & (p-1,3,6)
    \end{array}
    \right),
$$
$$
A_7=\left(
    \begin{array}{cccc}
(0,1,3)  & (0,1,5)  & (0,7,3) & (0,7,5)  \\
(p-1,5,6)  & (p-1,5,2)  & (1,3,6) & (1,3,2)  \\
(1,2,7)  & (1,2,1)  & (p-1,6,7) & (p-1,6,1)
    \end{array}
    \right),
$$
$$
A_8=\left(
    \begin{array}{cccc}
(0,1,2)  & (0,1,6)  & (0,7,2) & (0,7,6)  \\
(p-1,5,1)  & (p-1,5,7)  & (1,3,1) & (1,3,7)  \\
(1,2,5)  & (1,2,3)  & (p-1,6,5) & (p-1,6,3)
    \end{array}
    \right),
$$
and
$$
A_9=\left(
    \begin{array}{cccc}
(0,1,1)  & (0,1,7)  & (0,7,1) & (0,7,7)  \\
(p-1,5,5)  & (p-1,5,3)  & (1,3,5) & (1,3,3)  \\
(1,2,2)  & (1,2,6)  & (p-1,6,2) & (p-1,6,6)
    \end{array}
    \right).
$$
Note that $A_1,A_2,\ldots,A_9$ cover each element in the set $\{(i,j,k)\mid i\in \{0,1,p-1\}, j\in\Z_8\setminus\{0,4\},k\in\Z_8\setminus\{0,4\}\}$ exactly once, and all row sums and all column sums in $A_1,A_2,\ldots,A_9$ are $(0,0,0)$. Since all the elements in $\Z_{p}\setminus \{0,1,p-1\}$ can be partitioned into $(p-3)/2$ pairs, each of the form $\{x,-x\}$, we can create two rows for each of the nine $p\times 4$ arrays of an IMRS$^*_{G\setminus (H_1\cup H_2)}(p,4;9)$ as follows:
$$
B_1=\left(
    \begin{array}{cccc}
(x,1,1)  & (x,7,7)  & (-x,1,7) & (-x,7,1)  \\
(-x,7,7)  & (-x,1,1)  & (x,7,1) & (x,1,7)
    \end{array}
    \right),
$$
$$
B_2=\left(
    \begin{array}{cccc}
(x,1,2)  & (x,7,6)  & (-x,1,6) & (-x,7,2)  \\
(-x,7,6)  & (-x,1,2)  & (x,7,2) & (x,1,6)
    \end{array}
    \right),
$$
$$
B_3=\left(
    \begin{array}{cccc}
(x,1,3)  & (x,7,5)  & (-x,1,5) & (-x,7,3)  \\
(-x,7,5)  & (-x,1,3)  & (x,7,3) & (x,1,5)
    \end{array}
    \right),
$$
$$
B_4=\left(
    \begin{array}{cccc}
(x,2,1)  & (x,6,7)  & (-x,2,7) & (-x,6,1)  \\
(-x,6,7)  & (-x,2,1)  & (x,6,1) & (x,2,7)
    \end{array}
    \right),
$$
$$
B_5=\left(
    \begin{array}{cccc}
(x,2,2)  & (x,6,6)  & (-x,2,6) & (-x,6,2)  \\
(-x,6,6)  & (-x,2,2)  & (x,6,2) & (x,2,6)
    \end{array}
    \right),
$$
$$
B_6=\left(
    \begin{array}{cccc}
(x,2,3)  & (x,6,5)  & (-x,2,5) & (-x,6,3)  \\
(-x,6,5)  & (-x,2,3)  & (x,6,3) & (x,2,5)
    \end{array}
    \right),
$$
$$
B_7=\left(
    \begin{array}{cccc}
(x,3,1)  & (x,5,7)  & (-x,3,7) & (-x,5,1)  \\
(-x,5,7)  & (-x,3,1)  & (x,5,1) & (x,3,7)
    \end{array}
    \right),
$$
$$
B_8=\left(
    \begin{array}{cccc}
(x,3,2)  & (x,5,6)  & (-x,3,6) & (-x,5,2)  \\
(-x,5,6)  & (-x,3,2)  & (x,5,2) & (x,3,6)
    \end{array}
    \right),
$$
and
$$
B_9=\left(
    \begin{array}{cccc}
(x,3,3)  & (x,5,5)  & (-x,3,5) & (-x,5,3)  \\
(-x,5,5)  & (-x,3,3)  & (x,5,3) & (x,3,5)
    \end{array}
    \right).
$$
This procedure provides us an IMRS$^*_{G\setminus (H_1\cup H_2)}(p,4;9)$.

Furthermore, apply Construction \ref{con:Hole-1} with an IMRS$^*_{H_1\setminus H}(p,4,3)$ and an MRS$^*_{H_2}(p,4,4)$, which exist by Lemma \ref{lem:p,2,8}, to obtain an MRS$^*_{G}(p,4;16)$.
\end{proof}

\subsection{Construction III}

To present a new construction for incomplete magic rectangle sets, we need to define a special kind of arrays as follows.

Let $C_2$ be a $2\times 6$ array
$$
{\left(
    \begin{array}{cccccc}
3 & -3 & 2 & -2 & 1 & -1 \\
-3 & 3 & -2 & 2 & -1 & 1 \\
    \end{array}
    \right)}
$$
and $C_3$ be a $3\times 6$ array
$$
{\left(
    \begin{array}{cccccc}
3 & -3 & 2 & -2 & 1 & -1 \\
-2 & 2 & 1 & -1 & -3 & 3 \\
-1 & 1 & -3 & 3 & 2 & -2 \\
    \end{array}
    \right).}
$$
Let $p$ be an odd prime. Let $C_p$ be a $p\times 6$ array whose first 3 rows form a $C_3$ and the remaining $p-3$ rows are $(p-3)/2$ copies of $C_2$.

Each row of $C_p$ is a permutation of $\{1,2,3,-1,-2,-3\}$. The sum of each column of $C_p$ is $0$. Write $C_p=(c_{il})$, $1\leq i\leq p$ and $1\leq l\leq 6$. One can see that $c_{i,l}+c_{i,l+1}=0$ for any $1\leq i\leq p$ and any $l\in\{1,3,5\}$.

\begin{Construction}\label{con:C[p]}
Let $\alpha \geq0$ be an integer and $p$ be an odd prime. Let $H$ be a finite abelian group. Let $G\cong \Z_p\oplus H\oplus \Z_{2^{\alpha}}$ and $G_1\cong \Z_p\oplus H\oplus \Z_{2^{\alpha+3}}$. Let $G_2\cong \Z_p\oplus H\oplus \Z_{2^{\alpha+1}}$ be a subgroup of $G_1$. If there exists an MRS$^*_G(p,4;|G|/(4p))$, then there exists  an IMRS$^*_{G_1\setminus G_2}(p,4;6|G|/(4p))$.
\end{Construction}

\begin{proof}
Let $G_1=\{(w,x) \mid w\in\Z_p\oplus H, x\in\Z_{2^{\alpha+3}}\}$. Then $G_2=\{(w,4x) \mid w\in\Z_p\oplus H, x\in\Z_{2^{\alpha+1}}\}$ is a subgroup of $G_1$. Let $G=\{(w,u) \mid w\in\Z_p\oplus H, u\in\Z_{2^{\alpha}}\}$. Note that every element $(w,x)$ of $G_1\setminus G_2$ has a unique representation as $(w,x)=(w,8u+z)$, where $(w,u)\in G$ and  $z\in\{1,2,3,-1,-2,-3\}$.

Let $M_1,M_2,\ldots, M_{|G|/(4p)}$ be an MRS$^*_G(p,4;|G|/(4p))$, where $M_s=(m^s_{ij})$, $1\leq s\leq|G|/(4p)$, $1\leq i\leq p$ and $1\leq j\leq 4$. Write $m^s_{ij}=(b^s_{ij},a^s_{ij})$, where $b^s_{ij}\in \Z_p\oplus H$ and $a^s_{ij}\in \Z_{2^{\alpha}}$.

Let $C_p=(c_{il})$ be the $p\times 6$ array defined before Construction \ref{con:C[p]}, where $1\leq i\leq p$ and $1\leq l\leq 6$. For every $1\leq s\leq|G|/(4p)$ and $1\leq j\leq 4$, replace each column
$$((b^s_{1j},a^s_{1j}), (b^s_{2j},a^s_{2j}),\cdots,(b^s_{pj},a^s_{pj}))^T$$
of $M_s$ by 6 columns
$$((b^s_{1j},8a^s_{1j}+c_{1l}), (b^s_{2j},8a^s_{2j}+c_{2l}),\cdots,(b^s_{pj},8a^s_{pj}+c_{pl}))^T,$$
where $1\leq l\leq 6$. We use the triples $(s,j,l)$ to index these columns.

For every $1\leq s\leq |G|/(4p)$, based on $M_s$, we construct $6$ arrays of size $p\times 4$ as follows:
$$
M_{s1}=(
    \begin{array}{cccc}
(s,1,1)  & (s,2,2)    & (s,3,1)  & (s,4,2) 
    \end{array}
    ),
$$
$$
M_{s2}=(
    \begin{array}{cccc}
(s,1,2)  & (s,2,1)    & (s,3,2)  & (s,4,1) 
    \end{array}
    ),
$$
$$
M_{s3}=(
    \begin{array}{cccc}
(s,1,3)  & (s,2,4)    & (s,3,3)  & (s,4,4) 
    \end{array}
    ),
$$
$$
M_{s4}=(
    \begin{array}{cccc}
(s,1,4)  & (s,2,3)    & (s,3,4)  & (s,4,3) 
    \end{array}
    ),
$$
$$
M_{s5}=(
    \begin{array}{cccc}
(s,1,5)  & (s,2,6)    & (s,3,5)  & (s,4,6) 
    \end{array}
    ),
$$
and
$$
M_{s6}=(
    \begin{array}{cccc}
(s,1,6)  & (s,2,5)    & (s,3,6)  & (s,4,5) 
    \end{array}
    ).
$$
Then the set ${\cal M}=\{M_{sl}\mid 1\leq s\leq |G|/(4p),1\leq l\leq 6\}$ is an IMRS$^*_{G_1\setminus G_2}(p,4;6|G|/(4p))$.

Since each element of $G$ appears exactly once in $M_1,M_2,\cdots, M_{|G|/(4p)}$ and each row of $C_p$ is a permutation of $\{1,2,3,-1,-2,-3\}$, it is readily checked that each element of $G_1\setminus G_2$ appears exactly once in the arrays from $\cal M$. Since $c_{i,l}+c_{i,l+1}=0$ for any $1\leq i\leq p$ and any $l\in\{1,3,5\}$, the sum of each row  and the sum of each column in every array in $\cal M$ are both zero.
\end{proof}

\begin{Lemma}\label{lem:exp8}
Let $p$ be an odd prime. Let $S_2$ be a finite abelian and noncyclic $2$-group with $\exp(S_2)\geq 8$. Then there exists an MRS$^*_{\Z_{p}\oplus S_2}(p,4;|S_2|/4)$.
\end{Lemma}

\begin{proof}
We proceed by induction on $|S_2|$. If $|S_2|=16$, since $\exp(S_2)\geq 8$, we have $S_2\cong \Z_2\oplus \Z_8$. By Lemma \ref{lem:p,2,8}, there exists an MRS$^*_{\Z_{p}\oplus \Z_2\oplus \Z_8}(p,4,4)$ for any odd prime $p$.

If $|S_2|=32$, then $S_2$ is isomorphic to one of $\Z_4\oplus \Z_8$, $\Z_2\oplus \Z_2\oplus\Z_8$ and $\Z_2\oplus \Z_{16}$. By Lemma \ref{lem:p,4,8}, there exists an MRS$^*_{\Z_p\oplus \Z_4\oplus \Z_8}(p,4,8)$ for any odd prime $p$. Start from an MRS$^*_{\Z_p\oplus \Z_2\oplus \Z_2}(p,4,1)$, which exists by Lemma \ref{lem:p,2,2}, and then apply Construction \ref{con:C[p]} with $(\alpha,H)\in\{(0,\Z_2\oplus \Z_2),(1,\Z_2)\}$ to obtain an IMRS$^*_{\Z_p\oplus (S_2\setminus R)}(p,4,6)$ for $(S_2,R)\in \{(\Z_2\oplus \Z_2\oplus \Z_{8}, \Z_2\oplus \Z_2\oplus \Z_2), (\Z_2\oplus \Z_{16}, \Z_2\oplus \Z_4)\}$. Since an MRS$^*_{\Z_p\oplus R}(p,4,2)$ exists by Lemmas \ref{lem:p,2,2,2} and \ref{lem:p,2,4}, we can apply Construction \ref{con:Hole} to obtain an MRS$^*_{\Z_p\oplus S_2}(p,4,8)$.


If $|S_2|=64$, then up to isomorphism, $S_2\in\{(\Z_8)^2, \Z_2\oplus \Z_{32}, \Z_4\oplus \Z_{16}, (\Z_2)^2\oplus \Z_{16}, \Z_2\oplus \Z_4\oplus \Z_8, (\Z_2)^3\oplus \Z_8\}$. By Lemma \ref{lem:p,8,8}, there exists an MRS$^*_{\Z_p\oplus (\Z_8)^2}(p,4,16)$ for any odd prime $p$. Start from an MRS$^*_{\Z_p\oplus (\Z_2)^3}(p,4,2)$, which exists by Lemma \ref{lem:p,2,2,2}, and then apply Construction \ref{con:C[p]} with $(\alpha,H)\in\{(0,(\Z_2)^3),(1,(\Z_2)^2)\}$ to obtain an IMRS$^*_{\Z_p\oplus (S_2\setminus R)}(p,4,12)$ for $(S_2,R)\in \{((\Z_2)^3\oplus \Z_8, (\Z_2)^4), ((\Z_2)^2\oplus \Z_{16},(\Z_2)^2\oplus \Z_4 )\}$. Since an MRS$^*_{\Z_p\oplus R}(p,4,4)$ exists by Lemma \ref{lem:exp4}, we can apply Construction \ref{con:Hole} to obtain an MRS$^*_{\Z_p\oplus S_2}(p,4,16)$. Begin with an MRS$^*_{\Z_p\oplus \Z_2\oplus \Z_4}(p,4,2)$, which exists by Lemma \ref{lem:p,2,4}, and then apply Construction \ref{con:C[p]} with $(\alpha,H)\in\{(0,\Z_2\oplus \Z_4),(1,\Z_4)$, $(2,\Z_2)\}$ to obtain an IMRS$^*_{\Z_p\oplus (S_2\setminus R)}(p,4,12)$ for $(S_2,R)\in \{(\Z_2\oplus \Z_4\oplus \Z_8, \Z_2\oplus \Z_4\oplus \Z_2), (\Z_4\oplus \Z_{16}, \Z_4\oplus \Z_4), (\Z_2\oplus \Z_{32}, \Z_2\oplus \Z_8)\}$. Since an MRS$^*_{\Z_p\oplus R}(p,4,4)$ exists by Lemmas \ref{lem:exp4} and \ref{lem:p,2,8}, we can apply Construction \ref{con:Hole} to obtain an MRS$^*_{\Z_p\oplus S_2}(p,4,16)$.

If $|S_2|=2^n\geq 128$, then $S_2\cong N\oplus\Z_{2^\beta}$, where $\beta\geq3$, $\Z_{2^\beta}$ is isomorphic to a cyclic subgroup of $S_2$ with the largest order, and $N$ is a nontrivial abelian $2$-group with $\exp(N)\leq2^\beta$. If $\beta=3$, since $\exp(N)\leq8$ and $|N|=2^{n-3}\geq16$, $N$ is noncyclic. If $\beta\geq 4$, then $N\oplus\Z_{2^{\beta-3}}$ is noncyclic. Therefore, $N\oplus\Z_{2^{\beta-3}}$ for any $\beta\geq 3$ is an abelian and noncyclic $2$-group. If $\exp(N\oplus\Z_{2^{\beta-3}})\geq8$, then there exists an MRS$^*_{\Z_p\oplus N\oplus\Z_{2^{\beta-3}}}(p,4,2^{n-5})$ by induction hypothesis. If $\exp(N\oplus\Z_{2^{\beta-3}})\leq 4$, then there exists an MRS$^*_{\Z_p\oplus N\oplus\Z_{2^{\beta-3}}}(p,4,2^{n-5})$ by Lemma \ref{lem:exp4}.

Start from an MRS$^*_{\Z_p\oplus N\oplus\Z_{2^{\beta-3}}}(p,4,2^{n-5})$, and then apply Construction \ref{con:C[p]} with $(\alpha,H)=(\beta-3,N)$ to obtain an IMRS$^*_{\Z_p\oplus N\oplus(\Z_{2^\beta}\setminus \Z_{2^{\beta-2}})}$ $(p,4;3\cdot2^{n-4})$. If $\exp(N\oplus\Z_{2^{\beta-2}})\geq8$, then there exists an MRS$^*_{\Z_p\oplus N\oplus\Z_{2^{\beta-2}}}(p,4,2^{n-4})$ by induction hypothesis. If $\exp(N\oplus\Z_{2^{\beta-2}})\leq 4$, then there exists an MRS$^*_{\Z_p\oplus N\oplus\Z_{2^{\beta-2}}}(p,4,2^{n-4})$ by Lemma \ref{lem:exp4}. So there exists an MRS$^*_{\Z_p\oplus S_2}(p,4;2^{n-2})$ by Construction \ref{con:Hole}.
\end{proof}

Now one can combine the results of Lemmas \ref{lem:necessary}, \ref{thm:CH1}, \ref{lem:odd}, \ref{lem:exp4} and \ref{lem:exp8} to complete the proof of Theorem \ref{thm:main}.

\section{Concluding remarks}\label{sec:con}

This paper establishes the necessary and sufficient conditions for the existence of a $G$-magic rectangle set MRS$_G(a,b;c)$ for any finite abelian group $G$, thereby confirming Conjecture 5.1 presented by Cichacz and Hinc in \cite{CH21-1}.

Magic rectangle sets can be seen as a weak version of $3$-dimensional magic rectangles, which was introduced by Hagedorn \cite{Hage}. Let $a_1,a_2,\ldots,a_n$ be positive integers. An {\em $n$-dimensional magic rectangle} $n$-MR$(a_1,a_2,\ldots,a_n)$ is an $a_1\times a_2\times \cdots \times a_n$ array with entries $d_{i_1,i_2,\ldots,i_n}$ which are elements of $\{1,2,\ldots,a_1a_2\ldots a_n\}$, each appearing exactly once, such that all sums in the $k$-th direction, $1\leq k\leq n$, are equal to a constant $\sigma_k$. That is, for every $k$, $1\leq k\leq n$, $\sum_{j=1}^{a_k}d_{b_1,b_2,\ldots,b_{k-1},j,b_{k+1},\ldots,b_n}=\sigma_k$
for every selection of indices $b_1,b_2,\ldots,b_{k-1},b_{k+1},\ldots,b_n$ and $\sigma_k=a_k(1+a_1a_2\cdots a_n)/2$.

Compared with a $3$-dimensional magic rectangle, a magic rectangle set features each horizontal layer (an $a\times b$ array) with equal row sums and equal column sums, while vertical column sums are unrestricted. Although the necessary and sufficient conditions for the existence of a $3$-dimensional magic rectangle have been established in \cite[Theorem 6.1]{ZLZS}, it appears that no literature focuses on the existence of $3$-dimensional magic rectangles over finite groups. A further research direction is to examine the existence of $3$-dimensional magic rectangles over finite groups.

\end{document}